\documentclass{amsart}
\usepackage{amsmath}
\usepackage{amssymb}
\usepackage{amsthm}
\usepackage{amscd}
\usepackage{amsfonts}
\usepackage{graphicx}
\usepackage{fancyhdr}
\usepackage{epsfig}
\usepackage{setspace}
\usepackage{subfigure}
\usepackage{tikz}
\usepackage{pgf}
\usepackage{subfigure}
\usepackage{float}

\theoremstyle{plain} \numberwithin{equation}{section}
\newtheorem{theorem}{Theorem}[section]

\newtheorem{conjecture}[theorem]{Conjecture}
\newtheorem{lemma}[theorem]{Lemma}
\newtheorem{proposition}[theorem]{Proposition}

\theoremstyle{definition}
\newtheorem{definition}[theorem]{Definition}

\newtheorem{example}[theorem]{Example}

\DeclareMathOperator{\LCM}{LCM}
\DeclareMathOperator{\mi}{mi}
\DeclareMathOperator{\atoms}{atoms}
\DeclareMathOperator{\lcm}{lcm}

\title{Constructing monomial ideals with a given minimal resolution}
\author{Sonja Mapes}
\address{255 Hurley Hall; University of Notre Dame; Notre Dame, IN
 46556}
\email{smapes1@nd.edu}

\author{Lindsay C. Piechnik}
\address{High Point University, High Point, NC , 27262}
\email{piechnik.hpu@gmail.com}

\date{September, 21, 2015}

\begin{document}

\maketitle
\begin{abstract}
This paper gives a description of various recent results which
construct monomial ideals with a given minimal free resolution.  
We show that these are all instances of coordinatizing a
finite atomic lattice as found in \cite{mapes}.  Subsequently, we 
explain how in some of these cases (\cite{Faridi}, \cite{Floystad1}), where questions still remain,
this point of view can be applied.  We also prove an equivalence in
the case of trees between the notion of \emph{maximal} defined in
\cite{Floystad1} and a notion of being maximal in a Betti stratum.  
\end{abstract}

\section{Introduction}

In recent years there have been a number of papers (see
\cite{velascoFrames}, \cite{Faridi}, \cite{Floystad1}, \cite{katthan},
\cite{IKM-F}, and
\cite{NPS} for examples) where the authors 
focus on constructing monomial ideals with a specified minimal
resolution, typically described as being supported on a specific
CW-complex via the construciton in \cite{BS}.  Many of these constructions can be interpreted as
``coordinatizing'' a finite atomic lattice via the construction found
in \cite{mapes}.  The main purpose of this paper is
to bring attention to this fact through three particular cases as
found in \cite{velascoFrames}, \cite{Faridi}, and \cite{Floystad1}.  

In particular, there are still a number of
unanswered questions motivated by the later two papers.  We believe
considering the structure of the correspondng lcm lattices  will help 
 answer some of these questions.  Specifically,
this is because the lcm lattice of a monomial ideal encodes important data
which is obscured by only considering the cell complex which supports the
resolution.  More generally, we believe that for many questions concerning monomial
ideals it can be important to consider this additional data.  For example in the recent work found in \cite{FMS}, the
authors give a characterization of which finite atomic lattices can be
the lcm lattices of monomial ideals with pure resolutions.  For their
work, it is significant that they are working in the context of the
lcm lattice rather than a cell complex that supports the resolution.

This paper is structured as follows.  
First, in Section \ref{prelims} we present the necessary background on finite 
atomic lattices and coordinatizations.
Proposition \ref{whenCoord} in Section \ref{CharCoord}
completes the characterization of coordinatizations found in Theorem
3.2 in \cite{mapes}.  It should be noted that, at present, equivalent results
have been proven indpendently in \cite{IKM-F}.  However, we include our
proof here for completeness since the language is consistent with that of
\cite{mapes}.  

Section \ref{nearlyScarfFaridi} gives a description of
the ``nearly Scarf'' construction of \cite{velascoFrames}, and the
``minimal squarefree'' construction of \cite{phan} as
coordinatizations.  This section then shows the construction
found in \cite{Faridi} is also an example of a coordinatization.
Further,  we offer  context for how the ideals in
\cite{Faridi} fit in with the ``nearly Scarf'' ideals and the
``minimal squarefree'' ideals, and how this can be useful for
considering the questions posed in \cite{Faridi}. 

Finally in Section \ref{floystad}, we give a reformulation of the ideas in
\cite{Floystad1} in terms of the underlying lcm lattice of these
ideals.  We also show that for trees, there is an explict
description of the finite atomic lattice and coordinatization that
yields the ideals constructed in \cite{Floystad1}.  Using this
description we are able to show that for trees the \emph{maximal}
ideals construted in \cite{Floystad1} are also maximal in their Betti
stratum.  Further, we 
discuss of how these ideals could be useful for
understanding minimal resolutions of monomial ideals in general.

\section{Preliminaries}\label{prelims}

\subsection{Lattices}
A \emph{lattice} is a set $(P, <)$  with an order relation $<$, which is transitive and antisymmetric satisfying the following properties:
\begin{enumerate}
\item $P$ has a maximum element denoted by $\hat{1}$
\item $P$ has a minimum element denoted by  $\hat{0}$
\item Every pair of elements $a$ and $b$ in $P$ has a join $a \vee b$, which is the least upper bound of the two elements
\item Every pair of elements $a$ and $b$ in $P$ has a meet $a \wedge b$, which is the greatest lower bound of the two elements.  
\end{enumerate}

We define an \emph{atom} of a lattice $P$ to be an element $x \in P$ such that $x$ covers $\hat{0}$ (i.e. $x > \hat{0}$ and there is no element $a$ such that $x > a > \hat{0}$). We will denote the set of atoms as $\atoms (P)$.  
\begin{definition}
If $P$ is a lattice and every element in $P -\{\hat{0}\}$ is the join of atoms, then $P$ is an \emph{atomic lattice}.  Further, if $P$ is finite, then it is a \emph{finite atomic lattice}. 
\end{definition}

Given a lattice $P$,  elements $x \in P$ are \emph{meet-irreducible} if $x \neq a \wedge b$ for any $ a > x, b>x$. 
The set of meet-irreducible elements in $P$ is denoted by  $\mathrm{mi}(P)$. 
Given an element $x \in P$, the \emph {order ideal} of $x$ is  the set $\lfloor {x} \rfloor = \{a \in P | a \leqslant x\}$.  
Similarly,  the \emph{filter} of $x$ is $\lceil {x} \rceil = \{a \in P
| x \leqslant a\}$.  We can also define \emph{intervals} (open and
closed, respectively) between two
elements $a$ and $b$ of $P$ as
follows: $(a,b) = \{c \in P \, \vert \, a < c < b\}$ and $[a,b]= \{c \in P \, \vert \, a \leq c \leq b\}$. 

There are two different simplicial complexes that one can associate to a finite atomic meet-semilattice $P$ 
(or any poset for that matter).  
One is the \emph{order complex}, $\Delta(P)$, which is the complex 
whose vertices correspond to elements of $P$ and  facets correspond to maximal chains of $P$.  
For finite atomic lattices, like those discussed here,
we also have (a specific instance of) the \emph{cross cut complex}, denoted  $\Gamma(P)$. 
 In $\Gamma(P)$, vertices correspond to atoms of $P$ and simplices correspond to subsets of atoms which have a join or meet in $P$.  It is known that $\Delta(P)$ is homotopy equivalent to $ \Gamma(P)$ \cite{bjorner}.

\subsection{Coordinatizations}\label{coord}

One of the main results (Theorem 5.1) of \cite{phan} is to show that
every finite atomic lattice is in fact the lcm lattice of a monomial
ideal.  This result was generalized by a modified construction in
\cite{mapes}, which also showed that with the modified
construction  all monomial ideals can be realized this way.  We
include a brief discription of this work here for
the convenience of the reader.

Define a {\it labeling} of a finite atomic lattice $P$ as any assignment of non-trivial monomials $\mathcal {M} = \{m_{p_1}, ..., m_{p_t}\}$ to some set of elements $p_i \in P$.  It will be convenient to think of unlabeled elements as having the label $1$. Define the monomial ideal $M_{\mathcal{M}}$ to be the ideal generated by monomials
\begin{equation} \label{IdealDef}
x(a) = \prod_{ p \in \lceil{a}\rceil ^ c} m_p 
\end{equation} for each $a \in \atoms (P)$ where $\lceil a \rceil^c$ means take the complement of $\lceil a \rceil$ in $P$.  We say that the labeling $\mathcal{M}$ is a {\it coordinatization} if the lcm lattice of $M_{\mathcal{M}}$ is isomorphic to   $P$. 

The following theorem, which is Theorem 3.2 in \cite{mapes}, gives a
criteria for when a labeling is a coordinatization.  

\begin{theorem}\label{coordinatizations}
Any labeling $\mathcal{M}$ of elements in a finite atomic lattice $P$ by monomials satisfying the following two conditions will yield a coordinatization of $P$.

\begin{enumerate}
\item[(C1)] If $p \in \mi (P)$ then $m_p \not = 1$.  (i.e. all meet-irreducibles are labeled)
\item[(C2)] If $\gcd (m_p, m_q) \not = 1$ for some $p, q \in P$ then $p$ and $q$ must be comparable.  (i.e. each variable only appears in monomials along one chain in $P$.)
\end{enumerate}
\end{theorem}

Note that this theorem is not an ``if and only if'' statement, Section
\ref{CharCoord} addresses this issue.  A main ingreedient in Section
\ref{CharCoord} is the following discussion of deficit labelings from
\cite{mapes}.

To complete our introduction to coordinatizations, we  show that every monomial ideal is in fact a coordinatization of its lcm lattice.  Let $M$ be a monomial ideal with $n$ generators and let $L_M$ be its lcm lattice.  For notational purposes, let $L_M$ be the set consisting of elements denoted $l_{p}$ which represent the monomials occurring in $L_M$.  Now define the abstract finite atomic lattice $P$ where the elements in $P$ are formal symbols $p$ satisfying the relations $p < p'$ if and only if $l_{p} < l_{p}'$ in $L_M$.  In other words, $P$ is the abstract finite atomic lattice isomorphic to $L_M$ obtained by simply forgetting the data of the monomials in $L_M$.  Define a labeling of $P$ by letting $\mathcal{D}$ be the set consisting of monomials $m_p$ for each $p \in P$ where
\begin{equation}\label{deficitLabeling}
m_p =  \frac{\gcd \{ l_{t} \,| \, t > p\}}{l_{p}}.
\end{equation}
By convention $\gcd \{ l_{t} \,| \, t > p\}$ for $p = \hat{1}$ is defined to be $l_{\hat{1}}$.  Note that $m_p$ is a monomial since clearly $l_{p}$ divides $l_{t}$ for all $t >p$.

This labeling can be used to prove that every monomial ideal can be
realized as a coordinatization of its lcm lattice, as shown in the
following proposition, which appears as Proposition 3.6 in \cite{mapes}. 

\begin{proposition}\label{realizable}
Given $M$, a monomial ideal with lcm lattice $P_M$, if $P$ is the abstract finite atomic lattice such that $P$ and $P_M$ are isomorphic  as lattices, then the labeling $\mathcal{D}$ of $P$ as defined by (\ref{deficitLabeling}) is a coordinatization and the resulting monomial ideal $M_{P, \mathcal{D}} = M$.
\end{proposition}

\section{Characterizing coordinatizations}\label{CharCoord}

Theorem 3.2 in  \cite{mapes} gives a partial characterization of how to coordinatize a finite atomic lattice.  Further, \cite{mapes} explains how given a monomial ideal $M$, one can find the coordinatization of $L_M$ which produces $M$. Here we aim to use this process to characterize when a labeling is a coordinatization.
With this in mind, we introduce a construction similar to
$\mathcal{D}$, for a labeling $\mathcal{M}$ and a finite atomic
lattice $P$. This object, which we will show agrees with
$\mathcal{D}$ in the case where $\mathcal{M}$ is a coordinatization of
$P$, will be denoted  $\mathcal{D}_{\mathcal{M}}$.
Let $P$ be a finite atomic lattice with $n$ atoms, and $\mathcal{M}$ be any labeling of the lattice $P$. 
As in Section \ref{coord}, $M_{\mathcal{M}}$ will be the monomial
ideal generated by the monomials $x(a_i)$, described in (\ref{IdealDef}).
 
$\mathcal{D}_{\mathcal{M}}$ is constructed similarly to $\mathcal{D}$. However, rather than using the lcm lattice of $M_{\mathcal{M}}$, it uses the original lattice $P$ itself. Here
\begin{equation}\label{fakeLCMs1}
 l_{a_i} = x(a_i)
\end{equation}
 for each atom $a_i \in P$ and
\begin{equation}\label{fakeLCMs2}
 l_p =  \lcm \{ l_{t} \,| \, p > t \}
\end{equation}
for each element $ p \in P$.
 It is worth emphasizing that each of the $x(a_i)$ is used, appearing
 as $l_{a_i}$ for an atom of $P$, not just a minimial generating set
 for $M_{\mathcal{M}}$. The labeling $\mathcal{D}_{\mathcal{M}}$ is
 then defined as the set of monomials $m_p$ described in equation
 (\ref{deficitLabeling}), using the monomials $l_p$ defined in
 (\ref{fakeLCMs1}) and (\ref{fakeLCMs2}).  This means
 $\mathcal{D}_{\mathcal{M}}$ is a labeling of $P$ (remember that $P$
 may not be the lcm lattice of $M_{\mathcal{M}}$).

\begin{example}\label{DMexample}
In Figure \ref{twoLabelings}, we show a poset $P$ and the verticies which
are labeled using the variable $x$ in a labeling $\mathcal{M}$.  This example
 $\mathcal{M}$ violates condition (C2), since the
variable $x$ appears at non-comparable positions.  The
symbol $X$ indicates which elements of $P$ are labeled with $x$ in
$\mathcal{D}_{\mathcal{M}}$.  Note that one of the $x$
  labels ``moves'' to the minimal element in $P$. This is because
 there is at least one atom of $P$ that is not
  less than any of the poset elements labeled with $x$  in the original labeling $\mathcal{M}$.
   As a result,
  $x(a_i)$ will have a factor of $x$ for each atom $a_i$.  

We should also note that if one completed this partial labeling $\mathcal{M}$ so
that the other variables satisfied (C2) and  all the meet
irreducibles were nontrivially labeled (thus satisfying (C1)),  the resulting labeling would
be a coordinatization even though it does not satisfy (C2). 
\end{example}

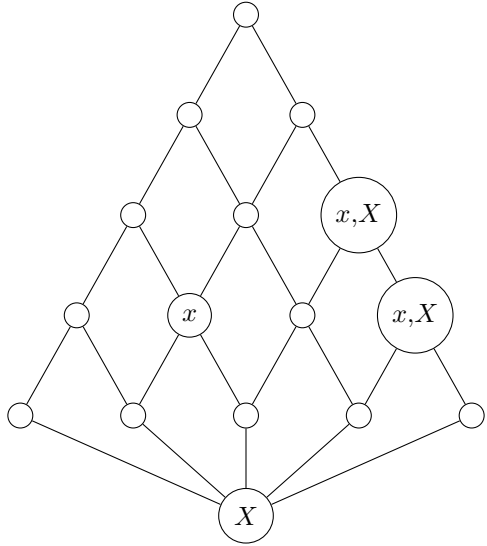
\begin{figure} 
\centering

\begin{tikzpicture}[scale=1, vertices/.style={draw, circle}]
             \node [vertices] (0) at (-0+0,0){$X$};
             \node [vertices] (1) at (-3+0,1.33333){};
             \node [vertices] (2) at (-3+1.5,1.33333){};
             \node [vertices] (3) at (-3+3,1.33333){};
             \node [vertices] (4) at (-3+4.5,1.33333){};
             \node [vertices] (5) at (-3+6,1.33333){};
             \node [vertices] (6) at (-2.25+0,2.66667){};
             \node [vertices] (7) at (-2.25+1.5,2.66667){$x$};
             \node [vertices] (8) at (-2.25+3,2.66667){}; 
             \node [vertices] (9) at (-2.25+4.5,2.66667){$x$,$X$};
             \node [vertices] (10) at (-1.5+0,4){};
             \node [vertices] (11) at (-1.5+1.5,4){};
             \node [vertices] (12) at (-1.5+3,4){$x$,$X$};
             \node [vertices] (13) at (-.75+0,5.33333){};
             \node [vertices] (14) at (-.75+1.5,5.33333){};
             \node [vertices] (15) at (-0+0,6.66667){};
     \foreach \to/\from in {0/1, 0/2, 0/3, 0/4, 0/5, 1/6, 2/6, 2/7, 3/8, 3/7, 4/8, 4/9, 5/9, 6/10, 7/10, 7/11, 8/12, 8/11, 9/12, 10/13, 11/13, 11/14, 12/14, 13/15, 14/15}
     \draw [-] (\to)--(\from);
     \end{tikzpicture}
\caption{The partial labeling described in example \ref{DMexample}, showing one variable from $\mathcal{M}$ and the same variable in $\mathcal{D}_{\mathcal{M}}$.}  \label{twoLabelings}
\end{figure}

The following proposition asserts that $\mathcal{D}_{\mathcal{M}}$
 determines whether the original labeling $\mathcal{M}$ was in fact a
 coordinatization of $P$.  An equivalent result
 appears independently as Theorem 3.4 in \cite{IKM-F} using different terminology.  We include our proof here
 as it continues in the language and terminology of Theorem 3.2 in
 \cite{mapes}.  In the proof we will use  $\deg_x{m}$ to
 denote the degree of a variable $x$  in a given monomial $m$.

\begin{proposition}\label{whenCoord}
 
 $\mathcal{M}$ is a coordinatization of $P$ if and only if $\mathcal{D}_{\mathcal{M}}$  satisfies (C1) and (C2).
In particular, this means $M_{\mathcal{M}}$ has an lcm lattice isomorphic to $P$.
\end{proposition}

\begin{proof}

 The forward direction follows from Proposition 3.6 of \cite{mapes}
 because if $\mathcal{M}$ is a coordinatization of $P$,  by definition
 $\mathcal{D}_{\mathcal{M}} = \mathcal{D}$ (i.e. $P =
 P_{M_{\mathcal{M}}}$).

 For the reverse direction we assume that $\mathcal{D}_{\mathcal{M}}$
 satisfies (C1) and (C2), which means that by Theorem \ref{coordinatizations},
 $\mathcal{D}_{\mathcal{M}}$ is a coordinatization.  In particular
 this tells us that $P$ is the lcm lattice of the monomial ideal
 $M_{\mathcal{D}}$ generated by the labeling
 $\mathcal{D}_{\mathcal{M}}$.  If we can show that 
 the monomial ideal $M_{\mathcal{D}}$ is equal to $M_{\mathcal{M}}$,
 then the lcm lattice of $M_{\mathcal{M}}$ will also be $P$, thus making $\mathcal{M}$ a coordinatization. 
 It is enough show that the monomial generators derived from  
 $\mathcal{M}$ agree with those obtained from $\mathcal{D}_{\mathcal{M}}$, specifically, that the exponents 
 on each variable agree. 

For clarity in the proof let us emphasize that in what follows the monomials $l_p$
will always be least common multiples of the generataors of
$M_{\mathcal{M}}$ and the monomials $m_p$ will always be the labelings
found in $\mathcal{M}_{\mathcal{D}}$.

 Let $x$ be a variable appearing in some generator of $M_{\mathcal{M}}$, 
  and let $r$ be the highest power of $x$ that divides any generator. 
 There is a subset of atoms in $P$ whose corresponding generators in $M_{\mathcal{M}}$ 
 have $x^r$ as a factor,
 call this set $\max(x)$. Define the set $A$ to be the set of elements in $P$ greater than
 or equal to the elements in $\max(x)$, and define the set $B$ to be
 the complement of $A$ in $P$. 

For each element $p$ in $A$ the $\deg_x{l_p}$ must be $r$, 
  and $\deg_x{l_p}$ for each element $p$ in $B$ must be strictly less than
  $r$.  In order to ensure that $\deg_x{x(a_i)}$ is $r$ for each $a_i$
  in $\max(x)$ where $x(a_i)$ is a generator of $M_{\mathcal{D}}$,
  it is enough to show that  \[\sum_{p \in B} \deg_x{m_p} \] is $r$. 

Note that
 $A$ must contain 
 the maximal element $\hat{1}$ in $P$, as it is the least common
 multiple of all generators. Also, 
  the minimal element $\hat{0}$, whose least common multiple is defined to be $1$, must be in $B$.

 We know $\hat{1}$ is in $A$, and $\deg_x{l_p}$ for all $p$ in $A$.
 Thus, $x^r$ divides $\gcd \{ l_{t} \,| \, t > b\}$ 
(from equation (\ref{deficitLabeling}))
 for every element $b$ which is the  maximal element of a chain in
 $B$.  
   If \[\sum_{p \in B-\{\hat{0}\}} \deg_x{m_p} = k < r, \] then
   $x^{r-k}$ will be a factor of $m_{\hat{0}}$.  
   To see this consider equation (\ref{deficitLabeling}) which shows that  $m_{\hat{0}} =  \gcd \{ l_{t} \,| \, t > \hat{0} \}$.
    This means the \[\sum_{p \in B} \deg_x{m_p} = r, \] as
    required. Therefore, $\deg_x{x(a_i)}$ is $r$ for each $a_i$ in
    $\max(x)$ where $x(a_i)$ is the corresponding generator of $M_{\mathcal{D}}$.
 
 It remains to show that exponents on $x$ agree for generators of $M_{\mathcal{D}}$ corresponding
 to atoms in $B$.
 For this, we consider the subposet $B$ of $P$, which is itself a poset, and apply our previous procedure iteratively. 
 Let $s < r$ be the highest power of $x$ coming from a generator of
 $M_{\mathcal{M}}$  corresponding to an atom in $B$. 
 The elements in $B$ greater than the set of atoms for which $x^s$ divides
 $x(a_i)$ in $M_{\mathcal{M}}$ will be $A_s$, 
 and the set of elements in $B - A_s$ will be $B_s$.

The monomial $l_p$ for each $p$ in $A_s$ has $x^s$ as a factor. This means if $A_s$ has
a unique maximal element it will be labeled with $x^{r-s}$ in
$\mathcal{D}_{\mathcal{M}}$ using equation (\ref{deficitLabeling}), and no other element $p$ in $A_s$ can have
 $x$ as a factor of the monomial $m_p$ in $\mathcal{D}_{\mathcal{M}}$.

  As in the previous case, there will be $s$ copies of $x$ remaining to label elements in
  $B_s$ in the construction of $\mathcal{D}_{\mathcal{M}}$.  Again,
  since $\hat{0}$ is in $B_s$, the \[\sum_{p \in B_s} \deg_x{m_p} \]
  will be $s$.  Therefore, $s$ is the $\deg_x{x(a_i)}$ for $a_i$ an
  atom in $A_s$ where the monomial $x(a_i)$ is a generator for
  $M_\mathcal{D}$.

  This process can be repeated for the next highest power of $x$ 
   appearing in a generator of $M_{\mathcal{M}}$ coming from $B_s$. 

 If, however, $A_s$ does not contain a unique maximal element,
 equation (\ref{deficitLabeling}) shows that there will be at least two non-comparable elements in $A_s$ which are labeled with $x^{r-s}$ in $\mathcal{D}_{\mathcal{M}}$. This means that $\mathcal{D}_{\mathcal{M}}$ contains copies of $x$ at non-comparable elements, hence violating (C2), which we are assuming to be true.
 
 Applying this procedure for each variable appearing in the generators
 of $M_{\mathcal{M}}$, shows that $M_\mathcal{D}$ and
 $M_{\mathcal{M}}$ have the same generators, making $\mathcal{M}$ a
 coordinatization of $P$.

\end{proof}

\begin{figure}
\center

\begin{tikzpicture}[scale=1, vertices/.style={draw, circle}]
             \node [vertices] (0) at (-0+0,0){};
             \node [vertices] (1) at (-3+0,1.33333){$A_3$};
             \node [vertices] (2) at (-3+1.5,1.33333){$A_2$};
             \node [vertices] (3) at (-3+3,1.33333){$A_1$};
             \node [vertices] (4) at (-3+4.5,1.33333){$A_1$};
             \node [vertices] (5) at (-3+6,1.33333){$A_1$};
             \node [vertices] (6) at (-2.25+0,2.66667){$A_3$};
             \node [vertices] (7) at (-2.25+1.5,2.66667){$A_2$};
             \node [vertices] (8) at (-2.25+3,2.66667){$A_1$};
             \node [vertices] (9) at (-2.25+4.5,2.66667){$A_1$};
             \node [vertices] (10) at (-1.5+0,4){$A_3$};
             \node [vertices] (11) at (-1.5+1.5,4){$A_2$};
             \node [vertices] (12) at (-1.5+3,4){$A_1$};
             \node [vertices] (13) at (-.75+0,5.33333){$A_3$};
             \node [vertices] (14) at (-.75+1.5,5.33333){$A_2$};
             \node [vertices] (15) at (-0+0,6.66667){$A_3$};
     \foreach \to/\from in {0/1, 0/2, 0/3, 0/4, 0/5, 1/6, 2/6, 2/7, 3/8, 3/7, 4/8, 4/9, 5/9, 6/10, 7/10, 7/11, 8/12, 8/11, 9/12, 10/13, 11/13, 11/14, 12/14, 13/15, 14/15}
     \draw [-] (\to)--(\from);
     \end{tikzpicture}
\caption{The sets $A_s$ for variable $x$, as described in example \ref{DMexamplePart2}.}\label{AsSets}
\end{figure}
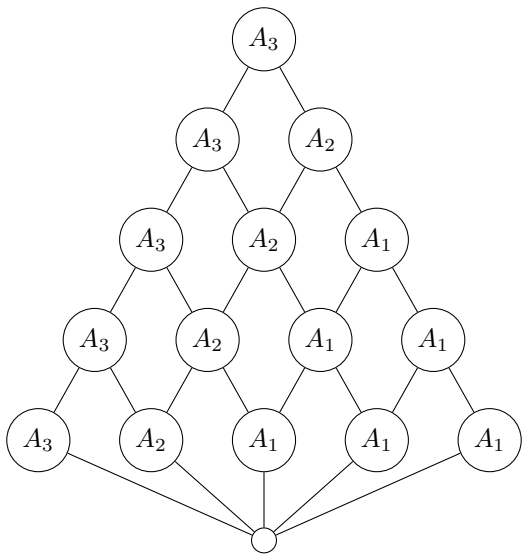

\begin{example} \label{DMexamplePart2}
For clarity we illustrate the sets $A_s$ for our Example
\ref{DMexample} in the Figure \ref{AsSets}.  Since the greatest
exponent of $x$ amongst the generators $x(a_i)$ for the labeling
$\mathcal{M}$ is 3, we use $A_3$ to denote the set $A$ from the proof,
 distinguishing it from the subsequent sets $A_2$ and $A_1$.  In this example,
 $B = A_2 \cup A_1 \cup \{\hat{0}\}$, $B_2 = A_1 \cup
\{\hat{0}\}$, and $B_1 = \{\hat{0}\}$.

\end{example}

\section{Nearly Scarf Ideals and Minimal monomial
  ideals} \label{nearlyScarfFaridi}
  
Coordinatizations of lattices have indirectly appeared in several other places as
instances of associating monomial ideals to cell complexes which then
support a minimal free resolution.  One
important example of coordinatizations are the ``nearly Scarf'' ideals
introduced by Peeva and Velasco in \cite{velascoFrames} and
\cite{velascoNonCW}.  The nearly Scarf construction is as follows.
Given a simplical complex $\Delta$, assign to each face $\sigma$ of
$\Delta$ the variable $x_\sigma$.  For a fixed vertex $v$ of $\Delta$,
let $A_\Delta(v)$ be the subcomplex of
$\Delta$ contating all the faces of $\Delta$ which do not contain the
vertex $v$.  The monomial ideal
$M_\Delta$ is generated by  the monomials \[m_v = \prod_{\sigma \in
  A_\Delta(v)} x_\sigma\] for each vertex $v$ of $\Delta$.

This construction can be seen as a coordinatization of the
(augmented) face poset of the simplicial complex $\Delta$ (note that the face
poset of a simplicial complex is a meet semi-lattice and so with a
maximal element becomes a finite atomic lattice).  Since
labeling every face $\sigma$ with a varaible $x_\sigma$ corresponds to
labeling every element of the face 
poset of the simplicial complex with a different variable, 
the formula for $m_v$ is identical to  equation (\ref{IdealDef}).  Clearly,
this satisfies the conditions of Proposition \ref{coordinatizations},
since all faces are labeled with distinct variables.

Another important example of coordinatizations are Phan's ``minimal
monomial ideals'' \cite{phan}.  In some sense these are the opposite
of nearly Scarf ideals, 
as they use the smallest number of variables possible.
  The construction for minimal squarefree ideals is
as follows.  Given a finite atomic lattice $P$,
let $\mi(P)$ denote the set of meet-irreducible elements in $P -
\{\hat{0}, \hat{1}\}$.  Then, label each element in $\mi(P)$ with a
distinct variable.  This labeling clearly satisfies conditions (C1)
and (C2), so it is a coordinatization of $P$.

\subsection{Resolutions supported on simplicial trees}

In \cite{Faridi}, Faridi addresses Scarf ideals corresponding to acyclic simplicial complexes, 
in particular  simplicial trees, offering
an alternative Scarf ideal in fewer variables than the ideals
constructed by Peeva and Velasco. Given a simplicial
complex $\Delta$, a variable $x_\sigma$ is still assigned to each face
$\sigma$, 
but only the variables for certain faces appear in the ideal.
 In \cite{Faridi},
 the monomial generators for each vertex $v$
 are defined as:
\begin{equation}\label{FaridiIdeal}
m'_v = \sqrt{ (\prod_{\{G \in B_\Delta(v)\}} x_{G - \{v\}} )
  (\prod_{\{F \in A_\Delta(v)\}} x_F(\prod_{\{| \sigma | = | F | -
    1\}} x_\sigma))}. 
\end{equation}  

Here,   $B_\Delta(v)$ and $A_\Delta(v)$  are the facets of $\Delta$
that do and do not contain $v$ respectively, and the square root
indicates that $m'_v$ is the square-free monomial containing all
variables in the described product. 

The first product indicates that for each facet $G$ containing $v$,
the variable for the facet of $G$ not containing $v$ is to be
included.  The second product indicates that for each facet $F$ not
containing $v$, both the variable for that face and all of the facets
of $F$ are to be included. 
Proposition 4.3 in \cite{Faridi} establishes that when $\Delta$ is
acyclic (a simplicial tree), the ideal generated by the monomials
$m'_v$ has a minimal free resolution supported on $\Delta$.

To see that this construction gives a coordinatization, we must check
that the meet-irreducibles are still labeled (since each face gets its
own variable we do not need to check condition (C2)), and that the products
that we obtain agree with the equations given in equation
(\ref{IdealDef}) for the given labeling.  

In the $B_\Delta(v)$ term, since  $G$ will always be a facet
and  we are taking products of variables corresponding to
$x_{G-\{v\}}$, as we let $v$ vary we use the labels on all of the
codimension one faces of each facet $G$.  Moreover, since every vertex
is contained in a facet,  we are using the labels on
every codimension one face of each of the facets of $\Delta$.

In the $A_\Delta(v)$ term,  we use the label on all of the facets $F$ which
do not contain $v$ and then all of the codimension one faces of that
facet $F$.  

As we let $v$ range over all verticies of $\Delta$, we see
that if every facet contains every vertex $v$ (i.e. if $G$ is in
$B_\Delta(v)$ for all $v$), then $\Delta = G$ and $\Delta$ is a
simplex.  Otherwise, each facet $G$ of $\Delta$ will be
in $B_\Delta(v)$ for some $v$ and then be in $A_\Delta(v')$ for some
$v'$.  So if $\Delta$ is not a simplex,  we must use the label on
every facet of $\Delta$.  If $\Delta$ is a simplex,  the only
facet corresponds to the maximal element of the face poset and its
labeling is irrelavent for the coordinatization as the element is
greater than every atom (which is consistent
with the fact that it will not appear as a monomial in equation (\ref{FaridiIdeal})).   

Therefore, we can describe the labeling of the (augmented) face poset $P_\Delta$
of $\Delta$ as labeling all of the elements corresponding to facets of
$\Delta$ (unless $\Delta$ is a simplex) and all of the elements corresponding to codimension one faces
of those facets. Call this labeling $\mathcal{F}_\Delta$.    

\begin{lemma}
The labeling $\mathcal{F}_\Delta$ of $P_{\Delta}$is a coordinatizataion.  
\end{lemma}   

\begin{proof}
Note that each element (or equivallently each face in $\Delta$) gets
labeled with a distinct variable.  Therefore, in $\mathcal{F}_\Delta$
condition (C1) will automatically be satisfied.  It remains to show
that $\mathcal{F}_\Delta$ non trivially labels all of the meet
irreducible elements of $P_\Delta$.

If $\Delta$ is a simplex,  $\mathcal{F}_\Delta$ labels all of the
coatoms of the face poset, which coincides with the set of meet
irreducibles. 
 
If $\Delta$ is not a simplex,  each facet $F$ of $\Delta$ is a simplex. So the interval of $P$ below $F$ is a boolean lattice. 
 The meet irreducibles of a boolean lattice are its co-atoms, 
 which in this sublattice correspond to codimension one faces of $F$.
  For each facet $F$, there is at least one vertex $v$ in $\Delta$ not contained in $F$.
  For each vertex $v$, the codimension one faces of facets not containing $v$ are labeled,
  so the codimension one faces of every facet are labeled.
   The only other possible meet irreducibles of the lattice would be
    elements corresponding to the facets themselves.
    Again, since for each facet $F$ there is at least one vertex $v$ of $\Delta$ 
    such that $F$ does not contain $v$, and for each vertex $v$ the facets 
    not containing $v$ are labeled, all facets are labeled.
   Therefore, all meet irreducible are labeled.
 
\end{proof}

\begin{lemma}
The monomial ideal created by the labeling $\mathcal{F}_\Delta$ equals
the ideal obtained using equation \ref{FaridiIdeal}.

\end{lemma}

\begin{proof}
Since both the ideal defined by $\mathcal{F}_\Delta$ and the ideal
defined by the equation \ref{FaridiIdeal} are squarefree, it suffices
to show that the variable $x$ divides $m'_v$ if and only if $x$
divides $x(a_v)$, where $a_v$ is the atom in $P_\Delta$ corresponding
to vertex $v$ in $\Delta$.  

Unpacking equation \ref{FaridiIdeal}, we see that if $x$ divides $m'_v$ then $x$ is the variable associated to either a
facet of $\Delta$ not containing $v$ or a codimension one face of any
facet of $\Delta$ which does not contain $v$.  These simplicies are
precisely the elements $P_\Delta$ which are not greater than $a_v$, and
are labeled via $\mathcal{F}_\Delta$.  So by the definition of
$x(a_v)$, $x$ will divide $x(a_v)$.  This equality
of sets also shows that if
$x$ divides $x(a_v)$, then $x$ will divide $m'_v$.
\end{proof}

In \cite{Faridi}, Faridi notes that her construction yields ideals
using fewer variables than in the nearly Scarf
construction.  It should be noted that her ideals are not in general the ``smallest possible."
 Phan's minimal
monomial ideals produce the ideals using the fewest variables \cite{phan}.
She
also considerers some ``in-between'' ideals, which are ideals where she adds back
some of the variables found in the nearly Scarf ideal to
some of the generators found using \ref{FaridiIdeal}.  
The reason she finds these
``in-between'' ideals have
different minimal resolutions is that adding back variables to some,
but not all, generators  will typically cause
the lcm lattice to change.  
 
We believe the perspective presented here, using lcm lattices and
coordinatiations, can shed important light on the questions posed
towards the end of \cite{Faridi}.

\section{Maximal Ideals with Resolutions supported on trees}\label{floystad}

In \cite{Floystad1}, Fl\o ystad  defines the category of monomial
ideals $M$ in a polynomial ring $S$ where the quotients $S/M$ are
Cohen-Macaulay, and he defines maximal elements in this category.  He
then gives constructions which associate maximal elements in this category to certain regular cell complexes 
(trees, and some polytopes), when minimal resolutions are supported on the cell complexes.  
Like in \cite{phan},  the focus is constructing monomial ideals with a specific cellular resolution.  
We will discuss the relationship between the two works in subsection
\ref{phanTies}, but first we provide a summary of the main
points from \cite{Floystad1} that will be used.

In \cite{Floystad1},  the set $CM(n,c)$ is defined as the set of ordered sets of $n$ monomials generating a monomial ideal $M$ such that the quotient ring is Cohen-Macaulay of codimension $c$.  
This is a category, but  the added structure is not necessary for our work here.  The set $CM_*(n,c)$ is the subset (subcategory) of $CM(n,c)$ consisting of monomial ideals which are squarefree and for which the sets \[V_t = \{i \,|\, x_t \mbox{ divides } m_i \} \subseteq [n]\] are distinct.  

 In \cite{Floystad1}, Fl\o ystad  initially defines what it means for a monomial ideal to be {\it maximal} using the maps in the category $CM(n,c)$.  
However, the maps in this category are heavily dependent on the
choice of coordinatization  for each monomial ideal.  So for our work,
his characterization identifying objects in $CM(n,c)$ with
 families $\mathcal{F}$ consisting of subsets of $[n]$ 
 is more useful.  These sets $\mathcal{F}$
correspond to the set of all sets $V_t$ described above. In
\cite{Floystad1}, he also gives a description of what properties a family of
sets $\mathcal{F}$ must have in order to correspond to an element in
$CM_*(n,c)$.  

The following are presented as Propositions $1.7$ and $1.10$ in \cite{Floystad1}.  Since they are equivalence statements we state them as definitions here to simplify language.

\begin{definition} \label{MaxFam}
A family of subsets of $[n]$, denoted $\mathcal{F}$, as described
above, which corresponds to an element in $CM_*(n,c)$, is {\it
  maximal} if it is reduced and is maximal among reduced families
corresponding to elements in $CM_*(n,c)$ for the refinement order.  A family $\mathcal{F}$ is {\it reduced} if it consists of elements which are not the disjoint union of other elements in $\mathcal{F}$.  The {\it refinement order} states that for two families of subsets $\mathcal{F} > \mathcal{G}$ if and only if $\mathcal{F}$ consists of refinements of elements of $\mathcal{G}$ together with additional subsets of $[n]$.   
\end{definition}     

In general, characterizing families $\mathcal{F}$ which are also in
$CM_*(n,c)$ seems to be a nontrivial task.  Fl\o ystad 
 restricts to families whose minimal resolution is supported on
a specific regular CW-complex.  These sets are defined as follows.

\begin{definition}\label{CMandRes}
Given a regular d-dimensional cell complex  $X$,    $CM_*(X)$ is the subset of $CM_*(n,c)$ whose minimal resolution is supported on $X$.
  
A family $\mathcal{F}$ is an object in $CM_*(X)$ if the following conditions hold. 
\begin{enumerate}
\item No $d$ of the subsets in $\mathcal{F}$ cover $[n]$. 
\item If $W$ is a union of subsets $\mathcal{F}$, the
  restriction of $X$ to the complement of $W$ is acyclic.
\item For every pair $F \varsubsetneq G$ of faces of $X$, there is a set $S \in \mathcal{F}$ such that $S \cap F$ is empty, but $S \cap G$ is not empty.
\end{enumerate}
\end{definition} 

 Definition \ref{CMandRes} describes how to ``label'' a regular cell complex $X$ so that we can construct an appropriate monomial ideal whose resolution is supported on $X$.  In particular, condition $1$ shows that the corresponding ideal has codimension at least $d+1$, condition $2$ guarantees that $X$ supports a cellular resolution, and condition $3$ ensures this resolution is minimal.

The following lemma which is Lemma 1.13 from \cite{Floystad1} gives a necessary condition for when a family $\mathcal{F}$ satisfying Definition \ref{MaxFam} is maximal.

\begin{lemma}\label{CMandMax}
If a family of subsets $\mathcal{F}$ of $[n]$ corresponds to a maximal object in $CM_*(X)$, then for every $S \in \mathcal{F}$  the restriction of $X$ to $S$ is connected.
\end{lemma}

 The families $\mathcal{F}$, which are used to describe ideals whose resolutions are supported on cell complexes in  \cite{Floystad1},  can be viewed as subsets of the lcm lattices of these monomial ideals.  What follows is a description of the connections between the constructions appearing in \cite{Floystad1} and  \cite{phan}. 

\subsection{Dictionary between labeling regular cell complexes and coordinatizing lattices}\label{phanTies}

We begin by addressing how to translate between the families $\mathcal{F}$ in  \cite{Floystad1} and  the lcm lattice associated to the ideal they represent.  
First consider the sets $V_t = \{i \,|\,  x_t \mbox{ divides } m_i \}$.  Let $M$ be the squarefree monomial ideal in $CM(n,c)$ corresponding  to a family $\mathcal{F} = \{V_1,..,V_s\}$.  
For each variable $x_t$,  there is a point in the deficit labeling, $p \in P = \LCM(M)$,  
such that $p$ is labeled with the variable $x_t$.  
So by the construction in \cite{phan},  $x_t$ will divide precisely the monomials that correspond to 
 $V_t = \lfloor{p}\rfloor^c \cap \atoms (P)$, where
 $\lfloor{p}\rfloor^c$  is the complement of the set of elements in $P$ which are less than
 or equal to $p$.  So reversing this, one can determine which element $p$ must
 be labeled with $x_t$, by taking the complement of $V_t$ in the set
 $\atoms (P)$, and then taking the join of these elements.  By
 definition, this labeling should yield the original ideal $M$.

\subsection{Codimension Two Cohen-Macaulay monomial \\ideals}

The Auslander-Buchsbaum formula makes it clear that the projective dimension of Cohen-Macaulay monomial ideals of codimension two must be two.  In terms of cellular resolutions, this implies their resolutions are supported on trees.  For this special case,  \cite{Floystad1} gives a very specific construction which associates to every tree $T$ a maximal monomial ideal in $CM(T)$ 
using any
given orientations of the edges of $T$.

The construction is as follows. Assign to every edge $e_i$ in $T$ two variables $x_i$ and $y_i$. 
 Delete the edge $e_i$ to produce two connected components of $T$,  $T_{i,1}$ and $T_{i,2}$.  The monomial associated to each vertex $v \in T$ is \[ m_v = (\prod_{\{i \,|\, v \in T_{i,1}\}} x_i ) (\prod_{\{i \,|\, v \in T_{i,2}\}} y_i). \]  
 The squarefree monomial ideal $M_T = (m_{v_1},\dots, m_{v_{n+1}})$ is maximal Cohen-Macaulay and its minimal resolution is supported on $T$.

\begin{example}\label{treeEx}
For each edge $e_{i}$ between two verticies $v_j$ and $v_k$, assume $j
<k$.  Let $T_{i,1}$ be the component of $T$ containing vertex $j$,
 and $T_{i,2}$ be the component of $T$ containing vertex $k$.
With this convention, the ideal obtained via the construction in  \cite{Floystad1}
to the tree in Figure \ref{treeFloystad} is   \[M_T =
(x_1 y_2 x_3 x_4,    y_1 x_2 x_3 x_4,    y_1y_2 x_3 x_4,     y_1 y_2
y_3 x_4,     y_1 y_2 y_3 y_4)\] in the ring $
k[x_1,x_2,x_3,x_4,y_1,y_2,y_3,y_4]$.
  
\begin{figure}\label{treeExVisual}
\center
\includegraphics[scale=0.5]{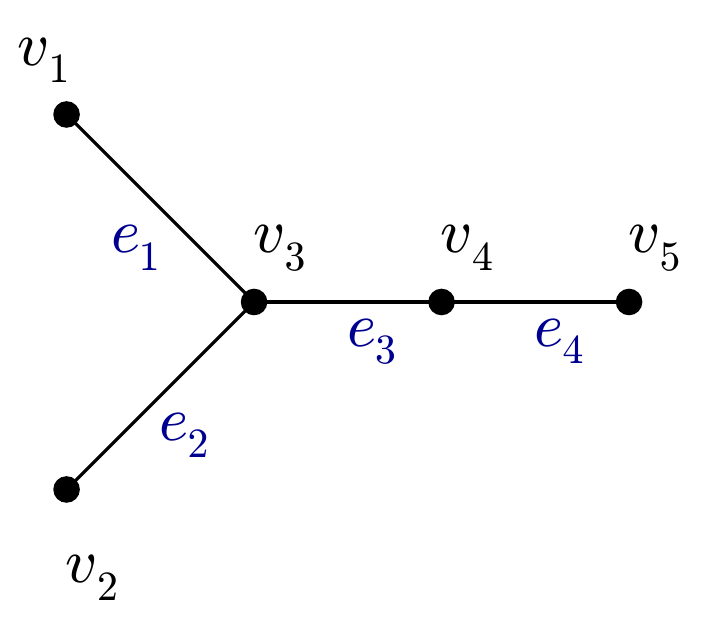}
\caption{The tree T described in example \ref{treeEx}.}\label{treeFloystad}
\end{figure}
\end{example}

We will show that for an appropriate choice of finite atomic lattice $P$, this construction coincides with the minimal squarefree coordinatization in \cite{phan}.  Given a tree $T$ with $n$ edges and $n+1$ vertices, define $P_T$ to be the set of all subtrees of $T$ ordered by inclusion (note that we include vertices and the empty set as subtrees).     

\begin{lemma}\label{PTlattice}
A poset $P_T$ defined as above is an element of $\mathcal{L}(n+1)$,
where $\mathcal{L}(n)$ is the set of all finite atomic lattices with
$n$ ordered atoms.
\end{lemma}

\begin{proof}
We will show that $P_T$ is a meet-semilattice with a maximal element
making it a finite lattice by Proposition 3.3.1 in \cite{Stanley}.   To show that $P_T$ is a meet-semilattice, we must show that for every pair of elements $a, b \in P_T$ there exists a meet or least upper bound.  Since $a$ and $b$ are subtrees of $T$, define $a \wedge b$ to be the intersection of $a$ and $b$.  Since $\emptyset \in P_T$, $a \cap b$ will be a subtree of $T$, so $P_T$ is a meet-semilattice.

It remains to show that $P_T$ is atomic with $n+1$ atoms.  This follows from the fact that $T$ has $n+1$ vertices and every subtree can be realized as an induced graph on a subset of the vertices.  
\end{proof}

The following proposition demonstrates that any coordinatization of the lattice $P_T$ defined above will yield a monomial ideal in $Mon(T)$, which is the set of monomial ideals whose resolution is supported on $T$.  

\begin{lemma}
The minimal resolution of any coordinatization of $P_T$ is supported on $T$.  
\end{lemma}

\begin{proof}
Since $P_T$ will be the
lcm lattice for any coordinatization of $P_T$, by  Proposition 1.2 in \cite{BS}, it is enough to show that
$T_{\leqslant p}$ is acyclic for each $p \in P_T$. 

Each $p \in P_T$ corresponds to a subtree of $P_T$ ordered by
inclusion, so by construction $T_{\leqslant p}$ is  the subtree corresponding to $p$.  Since they are themselves trees, each subtree is acyclic, so $T$ supports the minimal resolution of any coordinatization of $P_T$.
\end{proof}

Finally we show that  minimal squarefree coordinatization of $P_T$ in \cite{phan}  always agrees with the ideals  constructed in  \cite{Floystad1}.

\begin{theorem}\label{maxTreeIdeal}
If $\mathcal{M}$ is the minimal squarefree coordinatization of $P_T$, then $M_{\mathcal{M}} \cong M_T$  
\end{theorem}

\begin{proof}
Recall that in the construction of $M_T$ we assigned a variable to each subtree $T_{i,1}$ or $T_{i,2}$ of $T$ obtained by deleting an edge $e_i$ of $T$.  So each component was assigned a variable, and each vertex $v$ was assigned a monomial (the product of the variables corresponding to the trees $T_{i,j}$ containing $v$).  

We must show that the trees $T_{i,j}$ obtained by deleting edges are precisely the meet-irreducibles of $P_T$ and explain how to coordinateize $P_T$ to obtain $M_T$.

Clearly, the meet-irreducibles of $P_T$ will be the subtrees of $T' \subset T$ that have only one subtree $T'' \subset T$ containing them which satisfy  
\begin{equation}\label{treeCondition}
 |\{e_i \in T'\}| +1  = |\{e_i \in T''\}|
 \end{equation}
  where the $e_i$ are the edges of a tree.

If $T'$ is obtained as above by deleting an edge $e_i$ (i.e. $T' = T_{i,1}$), then the only subtree $T''$ satisfying equation \ref{treeCondition} is $T'' = T' \cup e_i$ (i.e. $T''$ is obtained by adding edge $e_i$ to $T'$).  Since the only other edges one could add are in the other connected component,  to add an edge we would be forced to add  $e_i$ as well, which would violate equation \ref{treeCondition}.  So the meet-irreducibles are precisely the subtrees $T_{i,j}$ obtained by deleting edge $e_i$.

As stated above, we want to use a minimal squarefree coordinatization of $P_T$.
If we place variables carefully, it will be clear that $M_{P_{T}, \mathcal{M}} \cong M_T$.  Recall that for $M_T$ the variables $x_i$ were assigned to the trees $T_{i,1}$, and $y_i$'s were assigned to the trees $T_{i,2}$.  Moreover, note that if $v \in T_{i,1}$ it is necessarily not in $T_{i,2}$ and vice versa.  
So  the trees $T_{i,1}, T_{i,2}$ partition the vertices into two disjoint sets.  The monomial label for the construction of $M_T$ assigns to each vertex the product of the variables corresponding to the subtrees containing $v$.  In lattice language, the subtrees containing $v$ will be in $\lceil {a_v} \rceil$ where $a_v$ is the atom corresponding to the vertex $v$.  
For our coordinatization construction this is not what we want, since we take the product over the complement of the filter.  However, the complement of the filter consists precisely of the subtrees not containing $v$, so we can make the following coordinatization.

Let $\mathcal{M}$ label $P_T$ as follows.  If $p \in P_T$  is a meet-irreducible,  corresponding to a $T_{i,1}$, denote it as $p_{i1}$ and label it with $y_i$, similarly if $p$ corresponds to a $T_{i,2}$ denote it as $p_{i2}$ label it with an $x_i$.  We see that:  

\begin{equation*}
\begin{aligned}
x(a_v) &= \prod_{ p \in \lceil{a_v}\rceil ^ c} m_p \\
&= (\prod_{p_{i1} \in \lceil{a_v}\rceil ^ c} y_i )( \prod_{p_{i2} \in \lceil{a_v}\rceil^ c} x_i)\\ 
&= (\prod_{\{i \,|\, v \in T_{i,2}\}} y_i) (\prod_{\{i \,|\, v \in T_{i,1}\}} x_i )\\
&= m_v.
\end{aligned}
\end{equation*}

\end{proof}

\begin{example} \label{treeExPart2}
Figure \ref{PosetwithLabel} depicts the lattice $P_T$ and the minimal squarefree
coordinatization used in the proof of Theorem \ref{maxTreeIdeal} for
the tree in Figure \ref{treeFloystad}, as described in Example \ref{treeEx}.

\end{example}

\begin{figure}[H]
\center
\begin{tikzpicture}[scale=1, vertices/.style={draw, circle}]
             \node [vertices] (0) at (-0+0,0){};
             \node [vertices] (1) at (-3+0,1.33333){$y_1$};
             \node [vertices] (2) at (-3+1.5,1.33333){$y_2$};
             \node [vertices] (3) at (-3+3,1.33333){};
             \node [vertices] (4) at (-3+4.5,1.33333){};
             \node [vertices] (5) at (-3+6,1.33333){$x_4$};
             \node [vertices] (6) at (-2.25+0,2.66667){};
             \node [vertices] (7) at (-2.25+1.5,2.66667){};
             \node [vertices] (8) at (-2.25+3,2.66667){};
             \node [vertices] (9) at (-2.25+4.5,2.66667){$x_3$};
             \node [vertices] (10) at (-2.25+0,4){$y_3$};
             \node [vertices] (12) at (-2.25+1.5,4){};
             \node [vertices] (11) at (-2.25+3,4){};
             \node [vertices] (13) at (-2.25+4.5,4){};
             \node [vertices] (14) at (-1.5+0,5.33333){$y_4$};
             \node [vertices] (16) at (-1.5+1.5,5.33333){$x_2$};
             \node [vertices] (15) at (-1.5+3,5.33333){$x_1$};
             \node [vertices] (17) at (-0+0,6.66667){};
     \foreach \to/\from in {0/1, 0/2, 0/3, 0/4, 0/5, 1/6, 2/7, 3/8, 3/6, 3/7, 4/8, 4/9, 5/9, 6/10, 6/11, 7/12, 7/10, 8/12, 8/13, 8/11, 9/13, 10/14, 11/14, 11/15, 12/16, 12/14, 13/16, 13/15, 14/17, 15/17, 16/17}
     \draw [-] (\to)--(\from);
     \end{tikzpicture}
\caption{$P_T$ with the minimal squarefree labeling outlined in example \ref{treeExPart2}.} \label{PosetwithLabel}
\end{figure}
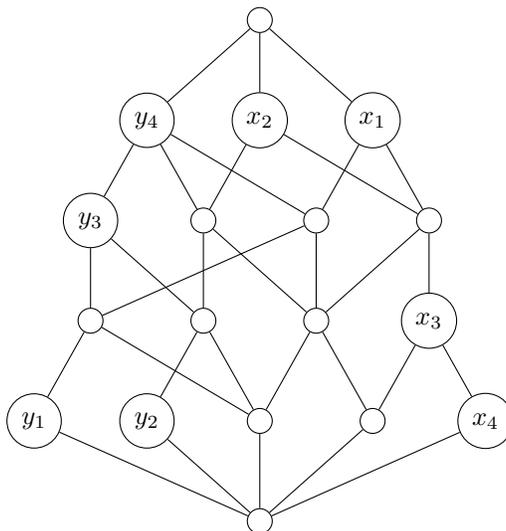

\subsection{Connection to Betti Strata}
  The sets $\mathcal{L}(n)$, introduced in
Lemma \ref{PTlattice},  have a rich structure studied
in \cite{mapes}.  Here we will highlight the important features 
 necessary for our discussion.  The most important
is that $\mathcal{L}(n)$ is itself a
finite atomic lattice (as shown in Theorem 4.2 of \cite{phan}) under the partial
order given by $P > Q$ if there is a join preserving map which is a
bijection on atoms from $P$ to $Q$.  Covering relations in
$\mathcal{L}(n)$ can be described as $P$ covers $Q$ if as sets $P= Q
\cup \{p\}$ where $p$ is a new element introduced to $Q$ with 
appropriate relations (Proposition 4.2 in \cite{mapes}). 
Since each element in $\mathcal{L}(n)$ is a
finite atomic lattice, and therefore can be assocated to a monomial
ideal, we can talk about the Betti numbers of these lattices as being
the Betti numbers of the associated monomial ideals.  Theorem 3.3 in
\cite{GPW}  guarantees that $\mathcal{L}(n)$ is stratified by
 total Betti numbers.  Understanding the boundaries
of these Betti strata in $\mathcal{L}(n)$ can provide insights for
how to move from one monomial ideal $M$ to another, whose minimal
resolution is easy to determine, in a way that  produces a
minimal resolution for $M$.  

The following proposition is a special case of Conjecture
\ref{maxEquals} in the case of trees.

\begin{proposition}
If $P$ is a lattice in $\mathcal{L}(n+1)$ satisfying $P > P_T$ then
the total Betti numbers of $P$ are greater than that of $P_T$
(i.e. $P$ is in a different Betti stratum than $P_T$).  
\end{proposition}

Before giving the proof, we introduce the following useful formulas from \cite{GPW}.
  One can compute the ``multigraded'' Betti numbers
for monomial ideals (or equivallently finite atomic lattices)
using intervals in the corresponding lcm lattice $P$.  Since the multidegree of a monomial
will always correspond to an element in the finite atomic lattice, we
abuse notation and say that the ``multidegree'' is an element $p$ in the
lattice $P$.  So the  computations  for graded and total Betti numbers are respectively as
follows:  $$b_{i,p} = \tilde{H}_{i-2}(\Gamma(\hat{0},p),k)$$ and  $$b_i = \sum_{p \in P}
\tilde{H}_{i-2}(\Gamma(\hat{0},p),k).$$

\begin{proof}
We need only consider the lattices $P$ in $\mathcal{L}(n+1)$
which cover $P_T$.  If we can show that for each of these lattices the
total Betti numbers are greater than that of $P_T$, we are done.
We know that these lattices $P$ only differ from the lattice
$P_T$ by one element, denoted $p$, so we need only consider how that one element
affects the Betti number computations.

First let us note that since
$T$ supports the minimal free resolution of any monomial ideal with
$P_T$ as the lcm lattice, the only elements
in $P_T$  for which $b_{i,p}$
are nonzero are the atoms and the elements covering the atoms.  

Now in $P$, which covers $P_T$, we know that all the elements $q$ from
$P_T$ where $b_{i,q}$ were nonzero will continue to be non zero since
they correspond to the face poset of $T$ which is simplicial, so they
are undisturbed by the addition of the element $p$.  This means the total
Betti numbers of $P$ are at least the total Betti numbers of $P_T$.
An obvious candidate for where a new non-zero Betti number might exist
is the element $p$ which has been added to $P_T$ to create $P$.  

Consider $\tilde{H}_0(\Gamma(\hat{0}, p),k)$.  If this is zero,
the order complex of the interval $(\hat{0}, p)$ in
$P$ is contractible.  However, since $p$ is the element we added to $P_T$,
all of the elements in $(\hat{0},p)$ correspond to subtrees of
$T$. Therefore, $\Gamma(\hat{0}, p)$ will be the union of the subtrees of
$T$ corresponding the elements covered by $p$.  If this is
contractible, then it should also be a subtree of $T$, and $p$
would already be an element of $P_T$.  So, $\Gamma(\hat{0}, p)$ is
not contractible, and $\tilde{H}_0(\Gamma(\hat{0}, p),k)$ is nonzero.
Meaning that the total Betti numbers (namely $b_2$) of $P$ are greater
than those of  $P_T$.           

\end{proof}

We conjecture that this should be true more generally.  Moreover, if
this conjecture is true, it offers an alternate (and
perhaps more useful) description of elements on the boundary of these
Betti strata.

\begin{conjecture}\label{maxEquals}
Let $X$ be a regular cell complex. The lcm lattice $P$ of a
maximal monomial ideal $M \in CM_*(X)$ satisfies the property that if
$Q > P$ in $\mathcal{L}(n)$, then the minimal resolution of $Q$ has
total Betti numbers greater than that of $P$.  In other words, $P$ is
maximal in its Betti stratum.  
\end{conjecture}

\end{document}